\documentclass[12pt,a4paper]{article}
\usepackage{amssymb}
\usepackage{amsmath}
\usepackage{amsfonts}

\newtheorem{theorem}{Theorem}
\newtheorem{proposition}{Proposition}

\newtheorem{corollary}{Corollary}
\newtheorem{definition}{Definition}
\newtheorem{example}{Example}

\newtheorem{remark}{Remark}

\newenvironment{proof}[1][Proof]{\noindent\textbf{#1.} }{\ \rule{0.0em}{0.0em}}
\input{tcilatex}
\begin{document}

\title{\textbf{Integration of Bicomplex Valued Function along Hyperbolic
Curve}}
\author{Chinmay Ghosh$^{1}$, Soumen Mondal$^{2}$ \\
%EndAName
$^{1}$Department of Mathematics\\
Kazi Nazrul University\\
Nazrul Road, P.O.- Kalla C.H.\\
Asansol-713340, West Bengal, India \\
chinmayarp@gmail.com \\
$^{2}$28, Dolua Dakshinpara Haridas Primary School\\
Beldanga, Murshidabad\\
Pin-742133\\
West Bengal, India\\
mondalsoumen79@gmail.com}
\date{}
\maketitle

\begin{abstract}
In this paper, we have defined bicomplex valued functions of bounded
variations and \ rectifiable hyperbolic path. We have studied the
integration of product-type bicomplex functions over rectifiable hyperbolic
path. Also we have established bicomplex analogue of the Fundamental Theorem
of Calculus for line integral.


\end{abstract}

\section{Introduction}

In $1883,$ Hamilton \cite{Ham}\ discovered four dimensional quaternions to
extend the complex number system to more than two dimensions. Quaternions
have algebraic properties of real and complex numbers except the
commutativity of multiplication. In $1892,$ C. Segre \cite{Seg} introduced
another four dimensional generalization of complex numbers, called bicomplex
numbers. Unlike quaternions the set of bicomplex numbers form a commutative
ring having divisors of zero. Following Segre, many mathematicians developed
the theory of functions on bicomplex numbers. The theory of bicomplex
variables is presented systematically in the book \cite{Pr} of G. B. Price.
There is an interesting subset of the set of bicomplex numbers, called the
set of hyperbolic numbers$.$ G. B. Price has not given much focus on
hyperbolic numbers. In \cite{Sob} we get a geometrical view of the
hyperbolic numbers.

In $2016,$ A. S. Balankin et al. \cite{Bal} introduced the concept of
hyperbolic intervals and the hyperbolic length of the hyperbolic interval.
In $2019,$ J. Bory-Rayes et al. \cite{Bo} introduced the integration of
product-type functions over hyperbolic curves using the concept of limit
over a filter base. In $2022,$ M. E. Luna-Elizarrar\'{a}s \cite{Lun1}
defined the partition of a hyperbolic interval and introduced the
integration of functions of hyperbolic variable.

In this paper, we have studied the integration of product-type bicomplex
function over hyperbolic path in a different way. Our results are presented
in the line of the book \cite{con}.

\section{Basic definitions}

We denote the set of real and complex numbers by $\mathbb{R}$ and $\mathbb{C}
$ respectively. We may think three imaginary numbers $\mathbf{i}_{1},\mathbf{%
i}_{2}$ and $\mathbf{j}$ governed by the rules%
\begin{equation*}
\mathbf{i}_{1}^{2}=-1,\mathbf{i}_{2}^{2}=-1,\mathbf{j}^{2}=1
\end{equation*}%
\begin{eqnarray*}
\mathbf{i}_{1}\mathbf{i}_{2} &=&\mathbf{i}_{2}\mathbf{i}_{1}=\mathbf{j} \\
\mathbf{i}_{1}\mathbf{j} &=&\mathbf{ji}_{1}=-\mathbf{i}_{2} \\
\mathbf{i}_{2}\mathbf{j} &=&\mathbf{ji}_{2}=-\mathbf{i}_{1}.
\end{eqnarray*}%
Then we have two complex planes $\mathbb{C}\left( \mathbf{i}_{1}\right)
=\left\{ x+\mathbf{i}_{1}y:x,y\in \mathbb{R}\right\} $ and $\mathbb{C}\left( 
\mathbf{i}_{2}\right) =\left\{ x+\mathbf{i}_{2}y:x,y\in \mathbb{R}\right\} ,$
both of which are identical to $\mathbb{C}.$ Bicomplex numbers(\cite{Alp},%
\cite{Sai}) are defined as $\zeta =z_{1}+\mathbf{i}_{2}z_{2}$ for $%
z_{1},z_{2}\in \mathbb{C}\left( \mathbf{i}_{1}\right) $. The set of all
bicomplex numbers is denoted by $\mathbb{BC}$. In particular if $%
z_{1}=x,z_{2}=\mathbf{i}_{1}y$ where $x,y\in \mathbb{R}$ we get $\zeta =x+%
\mathbf{j}y$ and these type of numbers are called hyperbolic numbers or
duplex numbers. The set of all hyperbolic numbers is denoted by $\mathbb{D}$%
. For $\left( z_{1}+\mathbf{i}_{2}z_{2}\right) ,\left( w_{1}+\mathbf{i}%
_{2}w_{2}\right) \in \mathbb{BC},$ the addition and multiplication are
definde as%
\begin{eqnarray*}
\left( z_{1}+\mathbf{i}_{2}z_{2}\right) +\left( w_{1}+\mathbf{i}%
_{2}w_{2}\right)  &=&\left( z_{1}+w_{1}\right) +\mathbf{i}_{2}\left(
z_{2}+w_{2}\right)  \\
\left( z_{1}+\mathbf{i}_{2}z_{2}\right) \left( w_{1}+\mathbf{i}%
_{2}w_{2}\right)  &=&\left( z_{1}w_{1}-z_{2}w_{2}\right) +\mathbf{i}%
_{2}\left( z_{1}w_{2}+z_{2}w_{1}\right) .
\end{eqnarray*}%
With these operations $\mathbb{BC}$ forms a commutative ring with zero
divisors. The elements $z_{1}+\mathbf{i}_{2}z_{2}\in \mathbb{BC}$ such that $%
z_{1}^{2}+z_{2}^{2}=0$ are the zero divisors. We denote the set of nonzero
zero divisors in $\mathbb{BC}$ by $\mathfrak{O}$ whereas $\mathfrak{O}_{0}=%
\mathfrak{O}\cup \left\{ 0\right\} $. On the other hand let us denote the
set of nonzero zero divisors in $\mathbb{D}$ by $\mathbb{O}$ whereas $%
\mathbb{O}_{0}=\mathbb{O}\cup \left\{ 0\right\} .$The interesting property
of a bicomplex number is its idempotent representation. Setting $\mathbf{e}%
_{1}=\frac{1+\mathbf{j}}{2}$ and $\mathbf{e}_{2}=\frac{1-\mathbf{j}}{2},$ we
get%
\begin{equation*}
z_{1}+\mathbf{i}_{2}z_{2}=\left( z_{1}-\mathbf{i}_{1}z_{2}\right) \mathbf{e}%
_{1}+\left( z_{1}+\mathbf{i}_{1}z_{2}\right) \mathbf{e}_{2}.
\end{equation*}%
Many calculations become easier in this representation.

The set of nonnegative hyperbolic numbers is%
\begin{equation*}
\mathbb{D}^{+}=\left\{ \nu _{1}\mathbf{e}_{1}+\nu _{2}\mathbf{e}_{2}:\nu
_{1},\nu _{2}\geq 0\right\} .
\end{equation*}%
A hyperbolic number $\zeta $ is said to be $\left( \text{strictly}\right) $
positive if $\zeta \in \mathbb{D}^{+}\backslash \left\{ 0\right\} .$The set
of nonnegative hyperbolic numbers is also defined as 
\begin{equation*}
\mathbb{D}^{+}=\left\{ x+y\mathbf{k}:x^{2}-y^{2}\geq 0,x\geq 0\right\} .
\end{equation*}

On the realization of $\mathbb{D}^{+},$\ M.E. Luna-Elizarraras et.al.\cite%
{Lun} defined a partial order relation on $\mathbb{D}$. For two hyperbolic
numbers $\zeta _{1},\zeta _{2}$ the relation $\preceq _{_{\mathbb{D}}}$ is
defined as 
\begin{equation*}
\zeta _{1}\preceq _{_{\mathbb{D}}}\zeta _{2}\text{ if and only if }\zeta
_{2}-\zeta _{1}\in \mathbb{D}^{+}.
\end{equation*}%
One can check that this relation is reflexive, transitive and antisymmetric.
Therefore $\preceq _{_{\mathbb{D}}}$ is a partial order relation on $\mathbb{%
D}$. This partial order relation $\preceq _{_{\mathbb{D}}}$ on $\mathbb{D}$
is an extension of the total order relation $\leq $ on $\mathbb{R}.$ We say $%
\zeta _{1}\prec _{_{\mathbb{D}}}\zeta _{2}$ if $\zeta _{1}\preceq _{_{%
\mathbb{D}}}\zeta _{2}$ but $\zeta _{1}\neq \zeta _{2}.$ Also we say $\zeta
_{2}\succeq _{_{\mathbb{D}}}\zeta _{1}$ if $\zeta _{1}\preceq _{_{\mathbb{D}%
}}\zeta _{2}$ and $\zeta _{2}\succ _{_{\mathbb{D}}}\zeta _{1}$ if $\zeta
_{1}\prec _{_{\mathbb{D}}}\zeta _{2}.$

\begin{definition}
\cite{Lun} For any hyperbolic number $\zeta =\nu _{1}\mathbf{e}_{1}+\nu _{2}%
\mathbf{e}_{2},$ the $\mathbb{D}-$modulus of $\zeta $ is defined by%
\begin{equation*}
\left\vert \zeta \right\vert _{\mathbb{D}}=\left\vert \nu _{1}\mathbf{e}%
_{1}+\nu _{2}\mathbf{e}_{2}\right\vert _{\mathbb{D}}=\left\vert \nu
_{1}\right\vert \mathbf{e}_{1}+\left\vert \nu _{2}\right\vert \mathbf{e}%
_{2}\in \mathbb{D}^{+}
\end{equation*}%
where $\left\vert \nu _{1}\right\vert $ and $\left\vert \nu _{2}\right\vert $
are the usual modulus of real numbers.
\end{definition}

\begin{definition}
\cite{Lun} A subset $A$\ of $\mathbb{D}$ is said to be $\mathbb{D}-$ bounded
if there exists $M\in \mathbb{D}^{+}$ such that $\left\vert \zeta
\right\vert _{\mathbb{D}}\preceq _{_{\mathbb{D}}}M$ for any $\zeta \in A.$
\end{definition}

Set%
\begin{eqnarray*}
A_{1} &=&\left\{ x\in \mathbb{R}:\exists \text{ }y\in \mathbb{R},\text{ }x%
\mathbf{e}_{1}+y\mathbf{e}_{2}\in A\right\} , \\
A_{2} &=&\left\{ y\in \mathbb{R}:\exists \text{ }x\in \mathbb{R},\text{ }x%
\mathbf{e}_{1}+y\mathbf{e}_{2}\in A\right\} .
\end{eqnarray*}%
If $A$\ is $\mathbb{D}-$bounded then $A_{1}$ and $A_{2}$ are bounded subset
of $\mathbb{R}$.

\begin{definition}
\label{D3}\cite{Lun} For a $\mathbb{D}-$ bounded subset $A$\ of $\mathbb{D},$
the supremum of $A$\ with respect to the $\mathbb{D}-$ modulus is defined by 
\begin{equation*}
\sup\nolimits_{\mathbb{D}}A=\sup A_{1}\mathbf{e}_{1}+\sup A_{2}\mathbf{e}%
_{2}.
\end{equation*}
\end{definition}

\begin{definition}
\cite{Lun} A sequence of hyperbolic numbers $\left\{ \zeta _{n}\right\}
_{n\geq 1}$ is said to be convergent to\textbf{\ }$\zeta \in \mathbb{D}$ if
for $\varepsilon \in \mathbb{D}^{+}\backslash \left\{ 0\right\} $ there
exists $k\in \mathbb{N}$ such that%
\begin{equation*}
\left\vert \zeta _{n}-\zeta \right\vert _{\mathbb{D}}\prec _{_{\mathbb{D}%
}}\varepsilon .
\end{equation*}%
Then we write 
\begin{equation*}
\lim\limits_{n\rightarrow \infty }\zeta _{n}=\zeta .
\end{equation*}
\end{definition}

\begin{definition}
\cite{Lun} A sequence of hyperbolic numbers $\left\{ \zeta _{n}\right\}
_{n\geq 1}$ is said to be $\mathbb{D}-$ Cauchy sequence\textbf{\ }$\zeta \in 
\mathbb{D}$ if for $\varepsilon \in \mathbb{D}^{+}\backslash \left\{
0\right\} $ $\exists $ $N\in \mathbb{N}$ such that%
\begin{equation*}
\left\vert \zeta _{N+m}-\zeta _{N}\right\vert _{\mathbb{D}}\prec _{_{\mathbb{%
D}}}\varepsilon 
\end{equation*}%
for all $m=1,2,3,...$ .
\end{definition}

Note that a sequence of hyperbolic numbers $\left\{ \zeta _{n}\right\}
_{n\geq 1}$ is convergent\textbf{\ }if and only if it is a $\mathbb{D}-$
Cauchy sequence\textbf{.}

\begin{definition}
\cite{Lun} A hyperbolic series $\sum\limits_{n=1}^{\infty }\zeta _{n}$ is
convergent if and only if its partial sum is a $\mathbb{D}-$ Cauchy
sequence, i.e.\textbf{,} for any $\varepsilon \in \mathbb{D}^{+}\backslash
\left\{ 0\right\} $ $\exists $ $N\in \mathbb{N}$ such that 
\begin{equation*}
\left\vert \sum\limits_{k=1}^{m}\zeta _{N+k}\right\vert _{\mathbb{D}}\prec
_{_{\mathbb{D}}}\varepsilon 
\end{equation*}%
for any $m\in \mathbb{N}.$
\end{definition}

\begin{definition}
\cite{Lun} A hyperbolic series $\sum\limits_{n=1}^{\infty }\zeta _{n}$ is $%
\mathbb{D}-$absolutely convergent if the series $\sum\limits_{n=1}^{\infty
}\left\vert \zeta _{n}\right\vert _{\mathbb{D}}$ is convergent.
\end{definition}

Every $\mathbb{D}-$ absolutely convergent series is convergent.

\begin{definition}
\cite{Bal} Let $\alpha =a_{1}\mathbf{e}_{1}+a_{2}\mathbf{e}_{2},\beta =b_{1}%
\mathbf{e}_{1}+b_{2}\mathbf{e}_{2}\in \mathbb{D}$ with $\alpha \preceq _{_{%
\mathbb{D}}}\beta .$ The closed hyperbolic interval ($\mathbb{D}-$interval)
is defined by 
\begin{equation*}
\left[ \alpha ,\beta \right] _{\mathbb{D}}=\left\{ \zeta \in \mathbb{D}%
:\alpha \preceq _{_{\mathbb{D}}}\zeta \preceq _{_{\mathbb{D}}}\beta \right\}
.
\end{equation*}%
Similarly the open hyperbolic interval ($\mathbb{D}-$interval) is defined by%
\begin{equation*}
\left( \alpha ,\beta \right) _{\mathbb{D}}=\left\{ \zeta \in \mathbb{D}%
:\alpha \prec _{_{\mathbb{D}}}\zeta \prec _{_{\mathbb{D}}}\beta \right\} .
\end{equation*}%
The length of the $\mathbb{D}-$interval $\left[ \alpha ,\beta \right] _{%
\mathbb{D}}$ or $\left( \alpha ,\beta \right) _{\mathbb{D}}$\ is defined by 
\begin{equation*}
l_{\mathbb{D}}\left( \left[ \alpha ,\beta \right] _{\mathbb{D}}\right)
=\beta -\alpha .
\end{equation*}%
$\left[ \alpha ,\beta \right] _{\mathbb{D}}$ is called a degenerate closed $%
\mathbb{D}-$interval if $\beta -\alpha $ is a nonnegative zero divisor
hyperbolic number and $\left[ \alpha ,\beta \right] _{\mathbb{D}}$ is called
a nondegenerate closed $\mathbb{D}-$interval if $\beta -\alpha $ $\in 
\mathbb{D}^{+}\backslash \mathbb{O}_{0}$.
\end{definition}

\begin{definition}
\cite{Bo} A set $A\mathbb{(\subset D)}$ is called product-type set if $A%
\mathbb{=}A_{1}\mathbf{e}_{1}+A_{2}e_{2}$ for some real sets $A_{1},A_{2}.$
\end{definition}

\begin{definition}
\cite{Lun1} Let $\left[ \alpha ,\beta \right] _{\mathbb{D}}$ be a
nondegenerate closed\ $\mathbb{D}-$interval. The partition $P$ of $\left[
\alpha ,\beta \right] _{\mathbb{D}}$ is the set $\left\{ \zeta _{0},\zeta
_{1},\zeta _{2},...,\zeta _{n}\right\} \subset \left[ \alpha ,\beta \right]
_{\mathbb{D}}$ such that%
\begin{equation*}
\alpha =\zeta _{0}\prec _{_{\mathbb{D}}}\zeta _{1}\prec _{_{\mathbb{D}%
}}\zeta _{2}\prec _{_{\mathbb{D}}}...\prec _{_{\mathbb{D}}}\zeta _{n}=\beta
\end{equation*}%
and 
\begin{equation*}
\zeta _{k}-\zeta _{k-1}\in \mathbb{D}^{+}\backslash \mathbb{O}%
_{0},k=1,2,...,n.
\end{equation*}
\end{definition}

A definition of infinity in the hyperbolic case is given in $\cite{Gh}$ as $%
\infty _{\mathbb{D}}=\infty \mathbf{e}_{1}+\infty _{\mathbb{D}}\mathbf{e}%
_{2}.$

\begin{definition}
\cite{Lun} A hyperbolic path or $\mathbb{D}-$path is a $\mathbb{D}-$%
continuous function $\Gamma :\left[ \alpha ,\beta \right] _{\mathbb{D}%
}\rightarrow \mathbb{BC}$, for some $\mathbb{D}-$interval $\left[ \alpha
,\beta \right] _{\mathbb{D}}$.
\end{definition}

In that case we get for $\tau =t\mathbf{e}_{1}+s\mathbf{e}_{2}\in \left[
\alpha ,\beta \right] _{\mathbb{D}},$ 
\begin{equation*}
\Gamma \left( \tau \right) =\gamma _{1}\left( t\right) \mathbf{e}_{1}+\gamma
_{2}\left( s\right) \mathbf{e}_{2}
\end{equation*}%
where $\gamma _{1}:\left[ a_{1},b_{1}\right] \rightarrow \mathbb{C},$ $%
\gamma _{2}:\left[ a_{2},b_{2}\right] \rightarrow \mathbb{C}$ are two paths
in $\mathbb{C}$ for $\alpha =a_{1}\mathbf{e}_{1}+a_{2}\mathbf{e}_{2}\in 
\mathbb{D},\beta =b_{1}\mathbf{e}_{1}+b_{2}\mathbf{e}_{2}\in \mathbb{D}$.

\begin{definition}
\cite{Da} A function $f:A\mathbb{=}A_{1}\mathbf{e}_{1}+A_{2}\mathbf{e}%
_{2}\subset \mathbb{BC\rightarrow BC}$ is called product-type if there exist 
$f_{i}:A_{i}\rightarrow 
%TCIMACRO{\U{2102} }%
%BeginExpansion
\mathbb{C}
%EndExpansion
$ for $i=1,2$ such that $f(\alpha _{1}e_{1}+\alpha _{2}e_{2})=f_{1}(\alpha
_{1})\mathbf{e}_{1}+f_{2}(\alpha _{2})\mathbf{e}_{2}$ for all $\alpha
_{1}e_{1}+\alpha _{2}e_{2}\in A$.
\end{definition}

\begin{example}
A $\mathbb{D-}$path $\Gamma (=\gamma _{1}e_{1}+\gamma _{2}e_{2}):\left[
\alpha ,\beta \right] _{\mathbb{D}}\rightarrow \mathbb{BC}$ is a
product-type function.
\end{example}

\section{Main Results}

In this section we prove our main results.

\begin{definition}
A function $\Gamma :\left[ \alpha ,\beta \right] _{\mathbb{D}}\rightarrow 
\mathbb{BC},$ for $\left[ \alpha ,\beta \right] _{\mathbb{D}}\subset \mathbb{%
D},$ is of $\mathbb{D}-$bounded variation if there exists $M\in \mathbb{D}%
^{+}$ such that for any partition $P=\left\{ \zeta _{0},\zeta _{1},\zeta
_{2},...,\zeta _{n}\right\} $ of $\left[ \alpha ,\beta \right] _{\mathbb{D}}$%
\begin{equation*}
v\left( \Gamma ;P\right) =\dsum\limits_{k=1}^{n}\left\vert \Gamma \left(
\zeta _{k}\right) -\Gamma \left( \zeta _{k-1}\right) \right\vert _{\mathbb{D}%
}\preceq _{_{\mathbb{D}}}M.
\end{equation*}%
The total $\mathbb{D}-$variation of $\Gamma ,$ $V\left( \Gamma \right) ,$ is
defined by%
\begin{equation*}
V\left( \Gamma \right) =\sup\nolimits_{\mathbb{D}}\left\{ v\left( \Gamma
;P\right) :P\text{ is a partition of }\left[ \alpha ,\beta \right] _{\mathbb{%
D}}\right\} .
\end{equation*}%
Obviously 
\begin{equation*}
V\left( \Gamma \right) \preceq _{_{\mathbb{D}}}M\prec _{_{\mathbb{D}}}\infty
_{\mathbb{D}}.
\end{equation*}
\end{definition}

\begin{proposition}
Let $\Gamma :\left[ \alpha ,\beta \right] _{\mathbb{D}}\rightarrow \mathbb{BC%
}$ be of $\mathbb{D}-$bounded variation.Then

$\left( a\right) $ If $P$ and $Q$ are partitions of $\left[ \alpha ,\beta %
\right] _{\mathbb{D}}$ and $P\subset Q$ then 
\begin{equation*}
v\left( \Gamma ;P\right) \preceq _{_{\mathbb{D}}}v\left( \Gamma ;Q\right) .
\end{equation*}

$\left( b\right) $ If $\Lambda :\left[ \alpha ,\beta \right] _{\mathbb{D}%
}\rightarrow \mathbb{BC}$ is also of $\mathbb{D}-$bounded variation and $%
a,b\in \mathbb{BC}$ then $a\Gamma +b\Lambda $ is of $\mathbb{D}-$bounded
variation and 
\begin{equation*}
V\left( a\Gamma +b\Lambda \right) \preceq _{_{\mathbb{D}}}\left\vert
a\right\vert _{\mathbb{D}}V\left( \Gamma \right) +\left\vert b\right\vert _{%
\mathbb{D}}V\left( \Lambda \right) .
\end{equation*}
\end{proposition}

\begin{proof}
$(a)$ Let $P=\left\{ \zeta _{0},\zeta _{1},\zeta _{2},...,\zeta _{n}\right\} 
$ be a partition of $\left[ \alpha ,\beta \right] _{\mathbb{D}}.$

First we examine the effect of adjoining one additional point $\eta $ to $P.$

The subinterval $[\zeta _{k-1},\zeta _{k}]_{\mathbb{D}}$ is divided into two
smaller subintervals $[\zeta _{k-1},\eta ]_{\mathbb{D}}$ and $[\eta ,\zeta
_{k}]_{\mathbb{D}}$ such that 
\begin{equation*}
\alpha =\zeta _{0}\prec _{_{\mathbb{D}}}\zeta _{1}\prec _{_{\mathbb{D}%
}}...\prec _{_{\mathbb{D}}}\zeta _{k-1}\prec _{_{\mathbb{D}}}\eta \prec _{_{%
\mathbb{D}}}\zeta _{k}\prec _{_{\mathbb{D}}}...\prec _{_{\mathbb{D}}}\zeta
_{n}=\beta ,
\end{equation*}

and 
\begin{equation*}
\eta -\zeta _{k-1};\zeta _{k}-\eta \in \mathbb{D}^{+}\backslash \mathbb{O}%
_{0}.
\end{equation*}

Then the set $P_{1}=\left\{ \zeta _{0},\zeta _{1},\zeta _{2},...,\zeta
_{k-1},\eta ,\zeta _{k},...,\zeta _{n}\right\} $ is a partition of $\left[
\alpha ,\beta \right] _{\mathbb{D}}$ such that $P\subset P_{1}.$

Now,%
\begin{equation*}
v\left( \Gamma ;P\right) =\left\vert \Gamma (\zeta _{1})-\Gamma (\zeta
_{0})\right\vert _{\mathbb{D}}+...+\left\vert \Gamma (\zeta _{k})-\Gamma
(\zeta _{k-1})\right\vert _{\mathbb{D}}+...+\left\vert \Gamma (\zeta
_{n})-\Gamma (\zeta _{n-1})\right\vert _{\mathbb{D}},
\end{equation*}

and%
\begin{equation*}
v\left( \Gamma ;P_{1}\right) =\left\vert \Gamma (\zeta _{1})-\Gamma (\zeta
_{0})\right\vert _{\mathbb{D}}+...+\left\vert \Gamma (\eta )-\Gamma (\zeta
_{k-1})\right\vert _{\mathbb{D}}+\left\vert \Gamma (\zeta _{k})-\Gamma (\eta
)\right\vert _{\mathbb{D}}+...+\left\vert \Gamma (\zeta _{n})-\Gamma (\zeta
_{n-1})\right\vert _{\mathbb{D}}.
\end{equation*}

Since%
\begin{eqnarray*}
\left\vert \Gamma (\zeta _{k})-\Gamma (\zeta _{k-1})\right\vert _{\mathbb{D}%
} &=&\left\vert \Gamma (\zeta _{k})-\Gamma (\eta )+\Gamma (\eta )-\Gamma
(\zeta _{k-1})\right\vert _{\mathbb{D}} \\
&\preceq &_{_{\mathbb{D}}}\left\vert \Gamma (\zeta _{k})-\Gamma (\eta
)\right\vert _{\mathbb{D}}+\left\vert \Gamma (\eta )-\Gamma (\zeta
_{k-1})\right\vert _{\mathbb{D}},
\end{eqnarray*}

it follows that 
\begin{equation*}
v\left( \Gamma ;P\right) \preceq _{_{\mathbb{D}}}v\left( \Gamma
;P_{1}\right) .
\end{equation*}

Since $Q$ can be obtained from $P$ by adjoining a finite number of
additional points to $P,$ one at a time, by repeating the arguments a finite
number of times, we have 
\begin{equation*}
v\left( \Gamma ;P\right) \preceq _{_{\mathbb{D}}}v\left( \Gamma ;Q\right) .
\end{equation*}

$(b)$ Let $\Omega (x)=a\Gamma (x)+b\Lambda (x),$ $x\in \left[ \alpha ,\beta %
\right] _{\mathbb{D}}.$

Let $P=\left\{ \zeta _{0},\zeta _{1},\zeta _{2},...,\zeta _{n}\right\} $ be
a partition of $\left[ \alpha ,\beta \right] _{\mathbb{D}}.$ Then%
\begin{equation*}
v\left( \Gamma ;P\right) =\left\vert \Gamma (\zeta _{1})-\Gamma (\zeta
_{0})\right\vert _{\mathbb{D}}+...+\left\vert \Gamma (\zeta _{k})-\Gamma
(\zeta _{k-1})\right\vert _{\mathbb{D}}+...+\left\vert \Gamma (\zeta
_{n})-\Gamma (\zeta _{n-1})\right\vert _{\mathbb{D}},
\end{equation*}%
\begin{equation*}
v\left( \Lambda ;P\right) =\left\vert \Lambda (\zeta _{1})-\Lambda (\zeta
_{0})\right\vert _{\mathbb{D}}+...+\left\vert \Lambda (\zeta _{k})-\Lambda
(\zeta _{k-1})\right\vert _{\mathbb{D}}+...+\left\vert \Lambda (\zeta
_{n})-\Lambda (\zeta _{n-1})\right\vert _{\mathbb{D}},
\end{equation*}%
\begin{equation*}
v\left( \Omega ;P\right) =\left\vert \Omega (\zeta _{1})-\Omega (\zeta
_{0})\right\vert _{\mathbb{D}}+...+\left\vert \Omega (\zeta _{k})-\Omega
(\zeta _{k-1})\right\vert _{\mathbb{D}}+...+\left\vert \Omega (\zeta
_{n})-\Omega (\zeta _{n-1})\right\vert _{\mathbb{D}}.
\end{equation*}

Now,%
\begin{eqnarray*}
\left\vert \Omega (\zeta _{r})-\Omega (\zeta _{r-1})\right\vert _{\mathbb{D}%
} &=&\left\vert a\Gamma (\zeta _{r})+b\Lambda (\zeta _{r})-a\Gamma (\zeta
_{r-1})-b\Lambda (\zeta _{r-1})\right\vert _{\mathbb{D}} \\
&\preceq &_{_{\mathbb{D}}}\left\vert a\right\vert _{\mathbb{D}}\left\vert
\Gamma (\zeta _{r})-\Gamma (\zeta _{r-1})\right\vert _{\mathbb{D}%
}+\left\vert b\right\vert _{\mathbb{D}}\left\vert \Lambda (\zeta
_{r})-\Lambda (\zeta _{r-1})\right\vert _{\mathbb{D}}
\end{eqnarray*}

Therefore%
\begin{equation*}
v\left( \Omega ;P\right) \preceq _{_{\mathbb{D}}}\left\vert a\right\vert _{%
\mathbb{D}}v\left( \Gamma ;P\right) +\left\vert b\right\vert _{\mathbb{D}%
}v\left( \Lambda ;P\right) .
\end{equation*}

Since $\Gamma $ and $\Lambda $ are functions of $\mathbb{D}-$bounded
variations on $\left[ \alpha ,\beta \right] _{\mathbb{D}},$ we have%
\begin{equation*}
v\left( \Gamma ;P\right) \preceq _{_{\mathbb{D}}}V(\Gamma ),
\end{equation*}

and%
\begin{equation*}
v\left( \Lambda ;P\right) \preceq _{_{\mathbb{D}}}V(\Lambda ),
\end{equation*}

for all partitions $P$ of $\left[ \alpha ,\beta \right] _{\mathbb{D}}.$

Therefore%
\begin{equation*}
v\left( \Omega ;P\right) \preceq _{_{\mathbb{D}}}\left\vert a\right\vert _{%
\mathbb{D}}V(\Gamma )+\left\vert b\right\vert _{\mathbb{D}}V(\Lambda ),
\end{equation*}

for all partitions $P$ of $\left[ \alpha ,\beta \right] _{\mathbb{D}}.$

This shows that%
\begin{eqnarray*}
\sup\nolimits_{\mathbb{D}}\left\{ v\left( \Omega ;P\right) :P\text{ is a
partition of }\left[ \alpha ,\beta \right] _{\mathbb{D}}\right\} &\preceq
&_{_{\mathbb{D}}}\left\vert a\right\vert _{\mathbb{D}}V(\Gamma )+\left\vert
b\right\vert _{\mathbb{D}}V(\Lambda ) \\
&\preceq &_{_{\mathbb{D}}}M,\text{ for some }M\in \mathbb{D}^{+}.
\end{eqnarray*}

Hence $\Omega (=a\Gamma +b\Lambda )$ is a function of $\mathbb{D}-$bounded
variation on $\left[ \alpha ,\beta \right] _{\mathbb{D}}$ and%
\begin{equation*}
V(a\Gamma +b\Lambda )\preceq _{_{\mathbb{D}}}\left\vert a\right\vert _{%
\mathbb{D}}V(\Gamma )+\left\vert b\right\vert _{\mathbb{D}}V(\Lambda ).
\end{equation*}
\end{proof}

\begin{definition}
A function $\Gamma :\left[ \alpha ,\beta \right] _{\mathbb{D}}\rightarrow 
\mathbb{BC}$ is $\mathbb{D}-$differentiable at $\tau =t\mathbf{e}_{1}+s%
\mathbf{e}_{2}\in \left[ \alpha ,\beta \right] _{\mathbb{D}}$ if 
\begin{equation*}
\lim\limits_{h\rightarrow 0,h\notin \mathbb{O}}\frac{\Gamma \left( \tau
+h\right) -\Gamma \left( \tau \right) }{h}
\end{equation*}%
exists in $\mathbb{D}$. Then we say%
\begin{equation*}
\Gamma ^{\prime }\left( \tau \right) =\lim\limits_{h\rightarrow 0,h\notin 
\mathbb{O}}\frac{\Gamma \left( \tau +h\right) -\Gamma \left( \tau \right) }{h%
},
\end{equation*}%
the $\mathbb{D}-$derivative of $\Gamma $ at $\tau .$
\end{definition}

A $\mathbb{D}-$path $\Gamma :\left[ \alpha ,\beta \right] _{\mathbb{D}%
}\rightarrow \mathbb{BC}$ is called $\mathbb{D}-$smooth if $\Gamma ^{\prime
}\left( \tau \right) $ exists for each $\tau \in \left[ \alpha ,\beta \right]
_{\mathbb{D}}$. Also $\Gamma $ is piecewise $\mathbb{D}-$smooth if there is
a partition $P=\left\{ \zeta _{0},\zeta _{1},\zeta _{2},...,\zeta
_{n}\right\} $ of $\left[ \alpha ,\beta \right] _{\mathbb{D}}$ such that $%
\Gamma $ is $\mathbb{D}-$smooth on each subinterval $\left[ \zeta
_{k-1},\zeta _{k}\right] .$

One can easily check that if $\gamma _{1},\gamma _{2}$ are (piecewise)
smooth the $\Gamma =\gamma _{1}\mathbf{e}_{1}+\gamma _{2}\mathbf{e}_{2}$ is
(piecewise) $\mathbb{D}-$smooth.

\begin{proposition}
\label{P2} Let $\Gamma (=\gamma _{1}\mathbf{e}_{1}+\gamma _{2}\mathbf{e}%
_{2}):\left[ \alpha ,\beta \right] _{\mathbb{D}}\rightarrow \mathbb{BC},$
for $\left[ \alpha ,\beta \right] _{\mathbb{D}}\subset \mathbb{D},$ is a $%
\mathbb{D-}$path. Then $\Gamma $ is of $\mathbb{D}-$bounded variation if and
only if $\gamma _{1},\gamma _{2}$ are of bounded variation. Also 
\begin{equation*}
V\left( \Gamma \right) =V\left( \gamma _{1}\right) \mathbf{e}_{1}+V(\gamma
_{2})\mathbf{e}_{2}.
\end{equation*}
\end{proposition}

\begin{proof}
Since $\Gamma (=\gamma _{1}\mathbf{e}_{1}+\gamma _{2}\mathbf{e}_{2}):\left[
\alpha ,\beta \right] _{\mathbb{D}}\rightarrow \mathbb{BC}$ is a $\mathbb{D-}
$path, for $i=1,2,$ $\gamma _{i}:[\alpha _{i},\beta _{i}]\rightarrow \mathbb{%
C}$ are paths in $\mathbb{C}$ where $\alpha =\alpha _{1}\mathbf{e}%
_{1}+\alpha _{2}\mathbf{e}_{2}$ and $\beta =\beta _{1}\mathbf{e}_{1}+\beta
_{2}\mathbf{e}_{2}.$

Let $P=\left\{ \zeta _{0},\zeta _{1},\zeta _{2},...,\zeta _{n}\right\} $ be
a partition of $\left[ \alpha ,\beta \right] _{\mathbb{D}}.$ Taking $\zeta
_{i}=\zeta _{i}^{1}\mathbf{e}_{1}+\zeta _{i}^{2}\mathbf{e}_{2},$ we get two
partitions $P_{1}=\left\{ \zeta _{0}^{1},\zeta _{1}^{1},\zeta
_{2}^{1},...,\zeta _{n}^{1}\right\} $ and $P_{2}=\left\{ \zeta
_{0}^{2},\zeta _{1}^{2},\zeta _{2}^{2},...,\zeta _{n}^{2}\right\} $ of $%
[\alpha _{1},\beta _{1}]$ and $[\alpha _{2},\beta _{2}]$ respectively$.$
Again for any two partitions $P_{1}$ of $[\alpha _{1},\beta _{1}]$ and $%
P_{2} $ of $[\alpha _{2},\beta _{2}]$, we can get a partition $P$ of $\left[
\alpha ,\beta \right] _{\mathbb{D}}.$

Now

\begin{eqnarray*}
v\left( \Gamma ;P\right) &=&\dsum\limits_{k=1}^{n}\left\vert \Gamma \left(
\zeta _{k}\right) -\Gamma \left( \zeta _{k-1}\right) \right\vert _{\mathbb{D}%
} \\
&=&\dsum\limits_{k=1}^{n}\left\vert \gamma _{1}(\zeta _{k}^{1})-\gamma
_{1}(\zeta _{k-1}^{1})\right\vert \mathbf{e}_{1}+\dsum\limits_{k=1}^{n}\left%
\vert \gamma _{2}(\zeta _{k}^{2})-\gamma _{2}(\zeta _{k-1}^{2})\right\vert 
\mathbf{e}_{2}
\end{eqnarray*}

Therefore%
\begin{equation}
v\left( \Gamma ;P\right) =v\left( \gamma _{1};P_{1}\right) \mathbf{e}%
_{1}+v\left( \gamma _{2};P_{2}\right) \mathbf{e}_{2}.  \label{1}
\end{equation}

Now it is clear from (\ref{1}) that for $M=M_{1}\mathbf{e}_{1}+M_{2}\mathbf{e%
}_{2}\in \mathbb{D}^{+}$ 
\begin{equation*}
v\left( \Gamma ;P\right) \preceq _{_{\mathbb{D}}}M\Leftrightarrow v\left(
\gamma _{1};P_{1}\right) \leq M_{1}\text{ and }v\left( \gamma
_{2};P_{2}\right) \leq M_{2}.
\end{equation*}

So, $\Gamma $ is of $\mathbb{D}-$bounded variation iff $\gamma _{1}$ and $%
\gamma _{2}$ are of bounded variation.

Also,%
\begin{eqnarray*}
V(\Gamma ) &=&\sup\nolimits_{\mathbb{D}}\left\{ v\left( \Gamma ;P\right) :P%
\text{ is a partition of }\left[ \alpha ,\beta \right] _{\mathbb{D}}\right\}
\\
&=&\sup\nolimits_{\mathbb{D}}\{v\left( \gamma _{1};P_{1}\right) \mathbf{e}%
_{1}+v\left( \gamma _{2};P_{2}\right) \mathbf{e}_{2}:P_{i}\text{ are
partitions of }[\alpha _{i},\beta _{i}],i=1,2\} \\
&=&V(\gamma _{1})\mathbf{e}_{1}+V(\gamma _{2})\mathbf{e}_{2},\text{ using
Definition \ref{D3}.}
\end{eqnarray*}
\end{proof}

\begin{proposition}
\label{P3} If $\Gamma :\left[ \alpha ,\beta \right] _{\mathbb{D}}\rightarrow 
\mathbb{BC}$ is piecewise $\mathbb{D}-$smooth then $\Gamma $ is of $\mathbb{D%
}-$bounded variation and 
\begin{equation*}
V\left( \Gamma \right) =\dint\limits_{\left[ \alpha ,\beta \right] _{\mathbb{%
D}}}\left\vert \Gamma ^{\prime }\left( \tau \right) \right\vert _{\mathbb{D}%
}d\tau 
\end{equation*}
\end{proposition}

\begin{proof}
Let $\Gamma =\gamma _{1}\mathbf{e}_{1}+\gamma _{2}\mathbf{e}_{2},$ $\alpha
=\alpha _{1}\mathbf{e}_{1}+\alpha _{2}\mathbf{e}_{2}$ and $\beta =\beta _{1}%
\mathbf{e}_{1}+\beta _{2}\mathbf{e}_{2}.$

Since $\Gamma :\left[ \alpha ,\beta \right] _{\mathbb{D}}\rightarrow \mathbb{%
BC}$ is piecewise $\mathbb{D}-$smooth, then $\gamma _{i}:[\alpha _{i},\beta
_{i}]\rightarrow 
%TCIMACRO{\U{2102} }%
%BeginExpansion
\mathbb{C}
%EndExpansion
$, for $i=1,2$ are piecewise smooth.

Then by Proposition $1.3(\cite{con},$ Chapter IV), $\gamma _{i}$ are of
bounded variation and 
\begin{equation*}
V(\gamma _{i})=\int\limits_{\alpha _{i}}^{\beta _{i}}\left\vert \gamma
_{i}^{^{\prime }}(\tau _{i})\right\vert d\tau _{i},\text{ for }i=1,2.
\end{equation*}

Then by Proposition \ref{P2}, $\Gamma $ is of $\mathbb{D}-$bounded variation
and 
\begin{equation*}
V(\Gamma )=V(\gamma _{1})\mathbf{e}_{1}+V(\gamma _{2})\mathbf{e}_{2}.
\end{equation*}

Since $\Gamma ^{^{\prime }}:\left[ \alpha ,\beta \right] _{\mathbb{D}%
}\rightarrow \mathbb{BC}$ is a $\mathbb{BC-}$function, then $\left\vert
\Gamma ^{^{\prime }}\right\vert _{\mathbb{D}}:\left[ \alpha ,\beta \right] _{%
\mathbb{D}}\rightarrow \mathbb{D}$ is a natural hyperbolic function (see $%
\cite{Tel}$) defined by $\left\vert \Gamma ^{^{\prime }}\right\vert _{%
\mathbb{D}}(\tau )=\left\vert \Gamma ^{^{\prime }}(\tau )\right\vert _{%
\mathbb{D}}$ for $\tau =\tau _{1}\mathbf{e}_{1}+\tau _{2}\mathbf{e}_{2}.$

Now%
\begin{eqnarray*}
\int\limits_{\lbrack \alpha ,\beta ]_{\mathbb{D}}}\left\vert \Gamma
^{^{\prime }}(\tau )\right\vert _{\mathbb{D}}d\tau &=&\left(
\int\limits_{\alpha _{1}}^{\beta _{1}}\left\vert \gamma _{1}^{^{\prime
}}(\tau _{1})\right\vert d\tau _{1}\right) \mathbf{e}_{1}+\left(
\int\limits_{\alpha _{2}}^{\beta _{2}}\left\vert \gamma _{2}^{^{\prime
}}(\tau _{2})\right\vert d\tau _{2}\right) \mathbf{e}_{2} \\
&=&V(\gamma _{1})\mathbf{e}_{1}+V(\gamma _{2})\mathbf{e}_{2} \\
&=&V(\Gamma ).
\end{eqnarray*}

So, we have 
\begin{equation*}
V\left( \Gamma \right) =\dint\limits_{\left[ \alpha ,\beta \right] _{\mathbb{%
D}}}\left\vert \Gamma ^{\prime }\left( \tau \right) \right\vert _{\mathbb{D}%
}d\tau
\end{equation*}
\end{proof}

\begin{theorem}
\label{T1} Let $\Gamma (=\gamma _{1}e_{1}+\gamma _{2}e_{2}):\left[ \alpha
,\beta \right] _{\mathbb{D}}\rightarrow \mathbb{BC}$ be product-type
function of $\mathbb{D}-$bounded variation and suppose that the product-type
function $f(=f_{1}e_{1}+f_{2}e_{2}):\left[ \alpha ,\beta \right] _{\mathbb{D}%
}\rightarrow \mathbb{BC}$ is $\mathbb{D-}$continuous. Then there is $I\in 
\mathbb{BC}$ such that for every $\epsilon _{\mathbb{D}}\succ _{\mathbb{D}}0$
there is a $\delta _{\mathbb{D}}\succ _{\mathbb{D}}0$ such that when $%
P=\left\{ \zeta _{0},\zeta _{1},\zeta _{2},...,\zeta _{n}\right\} $ be a
partition of $\left[ \alpha ,\beta \right] _{\mathbb{D}}$ with $\left\Vert
P\right\Vert _{\mathbb{D}}=\sup_{\mathbb{D}}\{l_{\mathbb{D}}\left( \left[
\zeta _{k-1},\zeta _{k}\right] _{\mathbb{D}}\right) :1\leq k\leq n\}\prec _{%
\mathbb{D}}\delta _{\mathbb{D}}$ then 
\begin{equation*}
\left\vert I-\sum\limits_{k=1}^{n}f(\tau _{k})[\Gamma (\zeta _{k})-\Gamma
(\zeta _{k-1})]\right\vert _{\mathbb{D}}\prec _{_{\mathbb{D}}}\epsilon _{%
\mathbb{D}}
\end{equation*}

for whatever choice of points $\tau _{k},$ $\zeta _{k-1}\preceq _{_{\mathbb{D%
}}}\tau _{k}\preceq _{_{\mathbb{D}}}\zeta _{k}.$
\end{theorem}

\begin{proof}
Since $\Gamma (=\gamma _{1}\mathbf{e}_{1}+\gamma _{2}\mathbf{e}_{2}):\left[
\alpha ,\beta \right] _{\mathbb{D}}\rightarrow \mathbb{BC}$ is a
product-type function of $\mathbb{D}-$bounded variation, by Proposition \ref%
{P2} for $i=1,2$ $\gamma _{i}:\left[ \alpha _{i},\beta _{i}\right]
\rightarrow 
%TCIMACRO{\U{2102} }%
%BeginExpansion
\mathbb{C}
%EndExpansion
$ are of bounded variation, where $\alpha =\alpha _{1}\mathbf{e}_{1}+\alpha
_{2}\mathbf{e}_{2}$ and $\beta =\beta _{1}\mathbf{e}_{1}+\beta _{2}\mathbf{e}%
_{2}.$

Also since $f(=f_{1}\mathbf{e}_{1}+f_{2}\mathbf{e}_{2}):\left[ \alpha ,\beta %
\right] _{\mathbb{D}}\rightarrow \mathbb{BC}$ is $\mathbb{D-}$continuous and
product-type function, for $i=1,2$ $f_{i}:\left[ \alpha _{i},\beta _{i}%
\right] \rightarrow 
%TCIMACRO{\U{2102} }%
%BeginExpansion
\mathbb{C}
%EndExpansion
$ are continuous functions.

Then by Theorem $1.4$ $\cite{con}$, for $i=1,2$ there exist $I_{i}\in 
%TCIMACRO{\U{2102} }%
%BeginExpansion
\mathbb{C}
%EndExpansion
$ such that for every $\epsilon _{i}>0$ there is a $\delta _{i}>0$ such that
when $P_{i}=\left\{ \zeta _{0}^{i},\zeta _{1}^{i},\zeta _{2}^{i},...,\zeta
_{n}^{i}\right\} $ are partitions of $\left[ \alpha _{i},\beta _{i}\right] $
with $\left\Vert P_{i}\right\Vert =\max \{(\zeta _{k}^{i}-\zeta
_{k-1}^{i}):1\leq k\leq n\}<\delta _{i}$ then 
\begin{equation*}
\left\vert I_{i}-\sum\limits_{k=1}^{n}f_{i}(\tau _{k}^{i})[\gamma _{i}(\zeta
_{k}^{i})-\gamma _{i}(\zeta _{k-1}^{i})]\right\vert <\epsilon _{i}
\end{equation*}

for whatever choice of points $\tau _{k}^{i},$ $\zeta _{k-1}^{i}\leq \tau
_{k}^{i}\leq \zeta _{k}^{i}.$

Let $I=I_{1}\mathbf{e}_{1}+I_{2}\mathbf{e}_{2},$ $\epsilon _{\mathbb{D}%
}=\epsilon _{i}\mathbf{e}_{1}+\epsilon _{2}\mathbf{e}_{2},$ $\delta _{%
\mathbb{D}}=\delta _{1}\mathbf{e}_{1}+\delta _{2}\mathbf{e}_{2},$ $\tau
_{k}=\tau _{k}^{1}\mathbf{e}_{1}+\tau _{k}^{2}\mathbf{e}_{2}$ and $\zeta
_{k}=\zeta _{0}^{i}\mathbf{e}_{1}+\zeta _{0}^{i}\mathbf{e}_{2}$ for $%
k=1,2,...,n.$

Then $P=\left\{ \zeta _{0},\zeta _{1},\zeta _{2},...,\zeta _{n}\right\} $ be
a partition of $\left[ \alpha ,\beta \right] _{\mathbb{D}}$ with $\left\Vert
P\right\Vert _{\mathbb{D}}=\left\Vert P_{1}\right\Vert \mathbf{e}%
_{1}+\left\Vert P_{2}\right\Vert \mathbf{e}_{2}\prec _{\mathbb{D}}\delta _{%
\mathbb{D}}$ and $\zeta _{k-1}\preceq _{\mathbb{D}}\tau _{k}\preceq _{%
\mathbb{D}}\zeta _{k}.$

Now%
\begin{eqnarray*}
\left\vert I-\sum\limits_{k=1}^{n}f(\tau _{k})[\Gamma (\zeta _{k})-\Gamma
(\zeta _{k-1})]\right\vert _{\mathbb{D}} &=&\left\vert
I_{1}-\sum\limits_{k=1}^{n}f_{1}(\tau _{k}^{1})[\gamma _{1}(\zeta
_{k}^{1})-\gamma _{i}(\zeta _{k-1}^{1})]\right\vert \mathbf{e}_{1} \\
&&+\left\vert I_{2}-\sum\limits_{k=1}^{n}f_{2}(\tau _{k}^{2})[\gamma
_{2}(\zeta _{k}^{2})-\gamma _{2}(\zeta _{k-1}^{2})]\right\vert \mathbf{e}_{2}
\\
&\prec &_{_{\mathbb{D}}}\epsilon _{i}\mathbf{e}_{1}+\epsilon _{i}\mathbf{e}%
_{2}=\epsilon _{\mathbb{D}}.
\end{eqnarray*}
\end{proof}

\begin{remark}
The number $I\in \mathbb{BC}$ of Theorem \ref{T1} is called the
Riemann-Stieljes $\mathbb{D}-$integral of $f$ with respect to $\Gamma $ over 
$\left[ \alpha ,\beta \right] _{\mathbb{D}}$ and is designated by%
\begin{equation*}
I=\int\limits_{\left[ \alpha ,\beta \right] _{\mathbb{D}}}fd\Gamma
=\int\limits_{\left[ \alpha ,\beta \right] _{\mathbb{D}}}f(\tau )d\Gamma
(\tau ).
\end{equation*}
\end{remark}

\begin{remark}
\label{R2} From the proof of Theorem \ref{T1} and Theorem $1.4$ $\cite{con}$%
, we can write%
\begin{equation*}
I=\int\limits_{\left[ \alpha ,\beta \right] _{\mathbb{D}}}fd\Gamma =\left(
\int\limits_{\alpha _{1}}^{\beta _{1}}f_{1}(t)d\gamma _{1}(t)\right) \mathbf{%
e}_{1}+\left( \int\limits_{\alpha _{2}}^{\beta _{2}}f_{2}(s)d\gamma
_{2}(s)\right) \mathbf{e}_{2}.
\end{equation*}
\end{remark}

The next result is very easy to prove, so we only state the result.

\begin{proposition}
Let $f$ and $g$ be two product-type bicomplex functions defined on $\left[
\alpha ,\beta \right] _{\mathbb{D}}$ and $\Gamma ,\Lambda :\left[ \alpha
,\beta \right] _{\mathbb{D}}\rightarrow \mathbb{BC}$ be product-type
functions of $\mathbb{D}-$bounded variation. Then for any $a,b\in \mathbb{BC}
$%
\begin{equation*}
(i)\text{ }\int\limits_{\left[ \alpha ,\beta \right] _{\mathbb{D}%
}}(af+bg)d\Gamma =a\int\limits_{\left[ \alpha ,\beta \right] _{\mathbb{D}%
}}fd\Gamma +b\int\limits_{\left[ \alpha ,\beta \right] _{\mathbb{D}%
}}gd\Gamma 
\end{equation*}%
\begin{equation*}
(ii)\text{ }\int\limits_{\left[ \alpha ,\beta \right] _{\mathbb{D}%
}}fd(a\Gamma +b\Lambda )=a\int\limits_{\left[ \alpha ,\beta \right] _{%
\mathbb{D}}}fd\Gamma +b\int\limits_{\left[ \alpha ,\beta \right] _{\mathbb{D}%
}}fd\Lambda .
\end{equation*}
\end{proposition}

\begin{proposition}
Let $\Gamma (=\gamma _{1}\mathbf{e}_{1}+\gamma _{2}\mathbf{e}_{2}):\left[
\alpha ,\beta \right] _{\mathbb{D}}\rightarrow \mathbb{BC}$ be product-type
function of $\mathbb{D}-$bounded variation and $f(=f_{1}\mathbf{e}_{1}+f_{2}%
\mathbf{e}_{2}):\left[ \alpha ,\beta \right] _{\mathbb{D}}\rightarrow 
\mathbb{BC}$ be $\mathbb{D-}$continuous product-type function. If $\alpha
=\zeta _{0}\prec _{_{\mathbb{D}}}\zeta _{1}\prec _{_{\mathbb{D}}}...\prec
_{_{\mathbb{D}}}\zeta _{k-1}\prec _{_{\mathbb{D}}}\zeta _{k}\prec _{_{%
\mathbb{D}}}...\prec _{_{\mathbb{D}}}\zeta _{n}=\beta ,$ then%
\begin{equation*}
\int\limits_{\left[ \alpha ,\beta \right] _{\mathbb{D}}}fd\Gamma
=\sum\limits_{k=1}^{n}\int\limits_{\left[ \zeta _{k-1},\zeta _{k}\right] _{%
\mathbb{D}}}fd\Gamma .
\end{equation*}
\end{proposition}

\begin{proof}
Let $\alpha =\alpha _{1}\mathbf{e}_{1}+\alpha _{2}\mathbf{e}_{2}$ , $\beta
=\beta _{1}\mathbf{e}_{1}+\beta _{2}\mathbf{e}_{2}$ and $\zeta _{k}=\zeta
_{k}^{1}\mathbf{e}_{1}+\zeta _{k}^{2}\mathbf{e}_{2}$ for $k=0,1,2,...,n.$

Now for $i=1,2$ $\gamma _{i}:\left[ \alpha _{i},\beta _{i}\right]
\rightarrow 
%TCIMACRO{\U{2102} }%
%BeginExpansion
\mathbb{C}
%EndExpansion
$ are of bounded variation and $f_{i}:\left[ \alpha _{i},\beta _{i}\right]
\rightarrow 
%TCIMACRO{\U{2102} }%
%BeginExpansion
\mathbb{C}
%EndExpansion
$ are continuous and also%
\begin{equation*}
\alpha _{i}=\zeta _{0}^{i}<\zeta _{1}^{i}<...<\zeta _{n}^{i}=\beta _{i}.
\end{equation*}

Then by Proposition $1.8\,\cite{con},$ we have for $i=1,2$%
\begin{equation*}
\int\limits_{\alpha _{i}}^{\beta _{i}}f_{i}d\gamma
_{i}=\sum\limits_{k=1}^{n}\int\limits_{\zeta _{k-1}^{i}}^{\zeta
_{k}^{i}}f_{i}d\gamma _{i}.
\end{equation*}

By Remark \ref{R2} we have%
\begin{eqnarray*}
\int\limits_{\left[ \alpha ,\beta \right] _{\mathbb{D}}}fd\Gamma &=&\left(
\int\limits_{\alpha _{1}}^{\beta _{1}}f_{1}(t)d\gamma _{1}(t)\right) \mathbf{%
e}_{1}+\left( \int\limits_{\alpha _{2}}^{\beta _{2}}f_{2}(s)d\gamma
_{2}(s)\right) \mathbf{e}_{2} \\
&=&\left( \sum\limits_{k=1}^{n}\int\limits_{\zeta _{k-1}^{1}}^{\zeta
_{k}^{1}}f_{1}(t)d\gamma _{1}(t)\right) \mathbf{e}_{1}+\left(
\sum\limits_{k=1}^{n}\int\limits_{\zeta _{k-1}^{2}}^{\zeta
_{k}^{2}}f_{2}(s)d\gamma _{2}(s)\right) \mathbf{e}_{2} \\
&=&\sum\limits_{k=1}^{n}\int\limits_{\left[ \zeta _{k-1},\zeta _{k}\right] _{%
\mathbb{D}}}fd\Gamma .
\end{eqnarray*}
\end{proof}

\begin{theorem}
If $\Gamma (=\gamma _{1}\mathbf{e}_{1}+\gamma _{2}\mathbf{e}_{2}):\left[
\alpha ,\beta \right] _{\mathbb{D}}\rightarrow \mathbb{BC}$ is piecewise $%
\mathbb{D-}$smooth and $f(=f_{1}\mathbf{e}_{1}+f_{2}\mathbf{e}_{2}):\left[
\alpha ,\beta \right] _{\mathbb{D}}\rightarrow \mathbb{BC}$ be $\mathbb{D-}$%
continuous product-type function, then%
\begin{equation*}
\int\limits_{\left[ \alpha ,\beta \right] _{\mathbb{D}}}fd\Gamma
=\int\limits_{\left[ \alpha ,\beta \right] _{\mathbb{D}}}f(\tau )\Gamma
^{\prime }(\tau )d\tau .
\end{equation*}
\end{theorem}

\begin{proof}
Since $\Gamma $ is piecewise $\mathbb{D-}$smooth, for $i=1,2$ $\gamma _{i}:%
\left[ \alpha _{i},\beta _{i}\right] \rightarrow 
%TCIMACRO{\U{2102} }%
%BeginExpansion
\mathbb{C}
%EndExpansion
$ are piecewise smooth, where $\alpha =\alpha _{1}\mathbf{e}_{1}+\alpha _{2}%
\mathbf{e}_{2}$ , $\beta =\beta _{1}\mathbf{e}_{1}+\beta _{2}\mathbf{e}_{2}.$

Also since $f$ is $\mathbb{D-}$continuous product-type function, $i=1,2$ $%
f_{i}:\left[ \alpha _{i},\beta _{i}\right] \rightarrow 
%TCIMACRO{\U{2102} }%
%BeginExpansion
\mathbb{C}
%EndExpansion
$ are continuous function.

Then by Theorem $1.9$ $\cite{con},$ for $i=1,2$%
\begin{equation}
\int\limits_{\alpha _{i}}^{\beta _{i}}f_{i}d\gamma _{i}=\int\limits_{\alpha
_{i}}^{\beta _{i}}f_{i}(t_{i})\gamma _{i}^{\prime }(t_{i})dt_{i}.  \label{2}
\end{equation}

Let $\tau =t_{1}\mathbf{e}_{1}+t_{2}\mathbf{e}_{2}.$ Then $\tau \in \left[
\alpha ,\beta \right] _{\mathbb{D}}$ and $\Gamma ^{\prime }(\tau )=\gamma
_{1}^{\prime }(t_{1})\mathbf{e}_{1}+\gamma _{2}^{\prime }(t_{2})\mathbf{e}%
_{2}.$

By Remark \ref{R2}%
\begin{eqnarray*}
\int\limits_{\left[ \alpha ,\beta \right] _{\mathbb{D}}}fd\Gamma &=&\left(
\int\limits_{\alpha _{1}}^{\beta _{1}}f_{1}d\gamma _{1}\right) \mathbf{e}%
_{1}+\left( \int\limits_{\alpha _{2}}^{\beta _{2}}f_{2}d\gamma _{2}\right) 
\mathbf{e}_{2} \\
&=&\left( \int\limits_{\alpha _{1}}^{\beta _{1}}f_{1}(t_{1})\gamma
_{1}^{\prime }(t_{1})dt_{1}\right) \mathbf{e}_{1}+\left( \int\limits_{\alpha
_{2}}^{\beta _{2}}f_{2}(t_{2})\gamma _{2}^{\prime }(t_{2})dt_{2}\right) 
\mathbf{e}_{2},\text{ by \ref{2}} \\
&=&\int\limits_{\left[ \alpha ,\beta \right] _{\mathbb{D}}}f(\tau )\Gamma
^{\prime }(\tau )d\tau .
\end{eqnarray*}
\end{proof}

If $\Gamma (=\gamma _{1}\mathbf{e}_{1}+\gamma _{2}\mathbf{e}_{2}):\left[
\alpha ,\beta \right] _{\mathbb{D}}\rightarrow \mathbb{BC}$ is a $\mathbb{D-}
$path, then the set $\{\Gamma (\tau ):\alpha \preceq _{_{\mathbb{D}}}\tau
\preceq _{_{\mathbb{D}}}\beta \}$ is called the trace of $\Gamma $ and is
denoted by $\{\Gamma \}.$ $\Gamma $ is a rectifiable $\mathbb{D-}$path if $%
\Gamma $ is a function of $\mathbb{D}-$bounded variation. For a partition $P$
of $\left[ \alpha ,\beta \right] _{\mathbb{D}},$ $v(\Gamma ;P)$ is the sum
of hyperbolic lengths of the line segment connecting points on the trace of $%
\Gamma .$ So $\Gamma $ is rectifiable if it has finite hyperbolic length and
its length is $V(\Gamma ).$ If $\Gamma $ is piecewise $\mathbb{D-}$smooth,
then $\Gamma $ is rectifiable and by Proposition \ref{P3}, its legth is $%
\dint\limits_{\left[ \alpha ,\beta \right] _{\mathbb{D}}}\left\vert \Gamma
^{\prime }\left( \tau \right) \right\vert _{\mathbb{D}}d\tau .$

If $\Gamma :\left[ \alpha ,\beta \right] _{\mathbb{D}}\rightarrow \mathbb{BC}
$ is a rectifiable $\mathbb{D-}$path with $\{\Gamma \}\subset \mathbb{%
E\subset BC}$ and $f:\mathbb{E\rightarrow BC}$ is $\mathbb{D-}$continuous
product-type function, then $f\circ \Gamma :\left[ \alpha ,\beta \right] _{%
\mathbb{D}}\rightarrow \mathbb{BC}$ is a $\mathbb{D-}$continuous
product-type function.

\begin{remark}
\label{R3} If $\Gamma (=\gamma _{1}\mathbf{e}_{1}+\gamma _{2}\mathbf{e}_{2}):%
\left[ \alpha ,\beta \right] _{\mathbb{D}}\rightarrow \mathbb{BC}$ is a $%
\mathbb{D-}$path, then $\{\Gamma \}=\{\gamma _{1}\}\mathbf{e}_{1}+\{\gamma
_{2}\}\mathbf{e}_{2}.$
\end{remark}

\begin{definition}
If $\Gamma :\left[ \alpha ,\beta \right] _{\mathbb{D}}\rightarrow \mathbb{BC}
$ is a rectifiable $\mathbb{D-}$path and $f$ is a product-type function
defined and $\mathbb{D-}$continuous on the trace of $\Gamma $ then the
(line) integral of $f$ along $\Gamma $ is 
\begin{equation*}
\dint\limits_{\left[ \alpha ,\beta \right] _{\mathbb{D}}}f(\Gamma (\tau
))d\Gamma (\tau ).
\end{equation*}
\end{definition}

This line integral is also denoted by 
\begin{equation*}
\int\limits_{\Gamma }f=\int\limits_{\Gamma }f(z)dz.
\end{equation*}

\begin{remark}
\label{R4} If $\Gamma (=\gamma _{1}\mathbf{e}_{1}+\gamma _{2}\mathbf{e}_{2}):%
\left[ \alpha ,\beta \right] _{\mathbb{D}}\rightarrow \mathbb{BC}$ is a
rectifiable $\mathbb{D-}$path and $f=(f_{1}\mathbf{e}_{1}+f_{2}\mathbf{e}%
_{2})$ is a product-type function defined and $\mathbb{D-}$continuous on $%
\{\Gamma \}$, then it is easy to verify that%
\begin{equation*}
\int\limits_{\Gamma }f=\left( \int\limits_{\gamma _{1}}f_{1}\right) \mathbf{e%
}_{1}+\left( \int\limits_{\gamma _{2}}f_{2}\right) \mathbf{e}_{2}.
\end{equation*}
\end{remark}

\begin{definition}
A function $\Phi :\left[ \lambda ,\mu \right] _{\mathbb{D}}\rightarrow \left[
\alpha ,\beta \right] _{\mathbb{D}}$ is said to be $\mathbb{D-}$monotone
function if any one of the following hold%
\begin{equation*}
i)\text{ }\Phi (\xi )\preceq _{_{\mathbb{D}}}\Phi (\tau )\text{ for any }\xi
,\tau \in \left[ \lambda ,\mu \right] _{\mathbb{D}}\text{ with }\xi \preceq
_{_{\mathbb{D}}}\tau ;
\end{equation*}%
\begin{equation*}
ii)\text{ }\Phi (\xi )\preceq _{_{\mathbb{D}}}\Phi (\tau )\text{ for any }%
\xi ,\tau \in \left[ \lambda ,\mu \right] _{\mathbb{D}}\text{ with }\xi
\succeq _{_{\mathbb{D}}}\tau .
\end{equation*}
\end{definition}

\begin{remark}
In the above definition if $(i)$ holds then $\Phi $ is said to be $\mathbb{D-%
}$monotone increasing function on $\left[ \lambda ,\mu \right] _{\mathbb{D}}$
and if $(ii)$ holds then $\Phi $ is said to be $\mathbb{D-}$monotone
decreasing function on $\left[ \lambda ,\mu \right] _{\mathbb{D}}.$
\end{remark}

\begin{remark}
If $\Phi =(\Phi _{1}\mathbf{e}_{1}+\Phi _{2}\mathbf{e}_{2}):\left[ \lambda
,\mu \right] _{\mathbb{D}}\rightarrow \left[ \alpha ,\beta \right] _{\mathbb{%
D}}$ is a $\mathbb{D-}$monotone increasing product-type function then for
each $i=1,2$ $\Phi _{i}:\left[ \lambda _{i},\mu _{i}\right] \rightarrow %
\left[ \alpha _{i},\beta _{i}\right] $ is monotone increasing function on $%
\left[ \lambda _{i},\mu _{i}\right] ,$ where $\lambda =\lambda _{1}\mathbf{e}%
_{1}+\lambda _{2}\mathbf{e}_{2},\mu =\mu _{1}\mathbf{e}_{1}+\mu _{2}\mathbf{e%
}_{2},\alpha =\alpha _{1}\mathbf{e}_{1}+\alpha _{2}\mathbf{e}_{2},\beta
=\beta _{1}\mathbf{e}_{1}+\beta _{2}\mathbf{e}_{2}.$
\end{remark}

If $\Gamma :\left[ \alpha ,\beta \right] _{\mathbb{D}}\rightarrow \mathbb{BC}
$ is a rectifiable $\mathbb{D-}$path and $\Phi :\left[ \lambda ,\mu \right]
_{\mathbb{D}}\rightarrow \left[ \alpha ,\beta \right] _{\mathbb{D}}$ is a $%
\mathbb{D-}$continuous, $\mathbb{D-}$monotone increasing function with $\Phi
(\left[ \lambda ,\mu \right] _{\mathbb{D}})=\left[ \alpha ,\beta \right] _{%
\mathbb{D}}$ (i.e., $\Phi (\lambda )=\alpha ,$ $\Phi (\mu )=\beta $) then $%
\Gamma \circ \Phi :\left[ \lambda ,\mu \right] _{\mathbb{D}}\rightarrow 
\mathbb{BC}$ is a $\mathbb{D-}$path such that $\{\Gamma \circ \Phi
\}=\{\Gamma \}.$ Also, if $\Phi (z)\notin \mathbb{O}$ for all $z\in \left[
\lambda ,\mu \right] _{\mathbb{D}},$ then $\Gamma \circ \Phi $ is
rectifiable because if $P=\left\{ \zeta _{0},\zeta _{1},\zeta _{2},...,\zeta
_{n}\right\} $ be a partition of $\left[ \lambda ,\mu \right] _{\mathbb{D}}$
then $P_{1}=\left\{ \Phi (\zeta _{0}),\Phi (\zeta _{1}),\Phi (\zeta
_{2}),...,\Phi (\zeta _{n})\right\} $ is a partition of $\left[ \alpha
,\beta \right] _{\mathbb{D}}.$ Therefore%
\begin{equation*}
\sum\limits_{k=1}^{n}\left\vert \Gamma (\Phi (\zeta _{k}))-\Gamma (\Phi
(\zeta _{k-1}))\right\vert _{\mathbb{D}}\preceq _{_{\mathbb{D}}}V(\Gamma )
\end{equation*}

so that $V(\Gamma \circ \Phi )\preceq _{\mathbb{D}}V(\Gamma )\prec _{\mathbb{%
D}}\infty _{\mathbb{D}}.$ So if $f$ is product-type $\mathbb{D-}$continuous
on $\,\{\Gamma \}=\{\Gamma \circ \Phi \}$ then $\int\limits_{\Gamma \circ
\Phi }f$ is well defined.

\begin{proposition}
If $\Gamma (=\gamma _{1}\mathbf{e}_{1}+\gamma _{2}\mathbf{e}_{2}):\left[
\alpha ,\beta \right] _{\mathbb{D}}\rightarrow \mathbb{BC}$ is a rectifiable 
$\mathbb{D-}$path and $\Phi =(\Phi _{1}\mathbf{e}_{1}+\Phi _{2}\mathbf{e}%
_{2}):\left[ \lambda ,\mu \right] _{\mathbb{D}}\rightarrow \left[ \alpha
,\beta \right] _{\mathbb{D}}$ is a $\mathbb{D-}$monotone increasing
product-type function with $\Phi (\lambda )=\alpha ,$ $\Phi (\mu )=\beta $
and $\Phi (z)\notin \mathbb{O}$ for all $z\in \left[ \lambda ,\mu \right] _{%
\mathbb{D}};$ then for any product-type $\mathbb{D-}$continuous function $%
f(=f_{1}\mathbf{e}_{1}+f_{2}\mathbf{e}_{2})$ on $\{\Gamma \}$%
\begin{equation*}
\int\limits_{\Gamma }f=\int\limits_{\Gamma \circ \Phi }f.
\end{equation*}
\end{proposition}

\begin{proof}
Let $\lambda =\lambda _{1}\mathbf{e}_{1}+\lambda _{2}\mathbf{e}_{2},\mu =\mu
_{1}\mathbf{e}_{1}+\mu _{2}\mathbf{e}_{2},\alpha =\alpha _{1}\mathbf{e}%
_{1}+\alpha _{2}\mathbf{e}_{2},\beta =\beta _{1}\mathbf{e}_{1}+\beta _{2}%
\mathbf{e}_{2}.$

Since $\Gamma (=\gamma _{1}\mathbf{e}_{1}+\gamma _{2}\mathbf{e}_{2}):\left[
\alpha ,\beta \right] _{\mathbb{D}}\rightarrow \mathbb{BC}$ is a rectifiable 
$\mathbb{D-}$path and $\Phi =(\Phi _{1}\mathbf{e}_{1}+\Phi _{2}\mathbf{e}%
_{2}):\left[ \lambda ,\mu \right] _{\mathbb{D}}\rightarrow \left[ \alpha
,\beta \right] _{\mathbb{D}}$ is a $\mathbb{D-}$monotone increasing
product-type function with $\Phi (\lambda )=\alpha ,$ $\Phi (\mu )=\beta ,$
then for $i=1,2$ $\gamma _{i}:[\alpha _{i},\beta _{i}]\rightarrow 
%TCIMACRO{\U{2102} }%
%BeginExpansion
\mathbb{C}
%EndExpansion
$ are rectifiable path and $\Phi _{i}:[\lambda _{i},\mu _{i}]\rightarrow
\lbrack \alpha _{i},\beta _{i}]$ are continuous increasing functions with $%
\Phi _{i}(\lambda _{i})=\alpha _{i}$ and $\Phi _{i}(\mu _{i})=\beta _{i}.$

Also since $f(=f_{1}\mathbf{e}_{1}+f_{2}\mathbf{e}_{2})$ is $\mathbb{D-}$%
continuous on $\{\Gamma \},$ then by Remark \ref{R3} we have $f_{i}$ are
continuous on $\{\gamma _{i}\}$ for $i=1,2.$

Then by Proposition $1.13$ ($\cite{con},$ Chapter $IV$), we have%
\begin{equation*}
\int\limits_{\gamma _{i}}f=\int\limits_{\gamma _{i}\circ \Phi _{i}}f,\text{
for }i=1,2.
\end{equation*}

Then by Remark \ref{R4}, we have%
\begin{eqnarray*}
\int\limits_{\Gamma }f &=&\left( \int\limits_{\gamma _{1}}f_{1}\right) 
\mathbf{e}_{1}+\left( \int\limits_{\gamma _{2}}f_{2}\right) \mathbf{e}_{2} \\
&=&\left( \int\limits_{\gamma _{1}\circ \Phi _{1}}f_{1}\right) \mathbf{e}%
_{1}+\left( \int\limits_{\gamma _{2}\circ \Phi _{2}}f_{2}\right) \mathbf{e}%
_{2} \\
&=&\int\limits_{\Gamma \circ \Phi }f.
\end{eqnarray*}
\end{proof}

Let $\Gamma (=\gamma _{1}\mathbf{e}_{1}+\gamma _{2}\mathbf{e}_{2}):\left[
\alpha ,\beta \right] _{\mathbb{D}}\rightarrow \mathbb{BC}$ is a rectifiable 
$\mathbb{D-}$path and for $\alpha \preceq _{_{\mathbb{D}}}\tau \preceq _{_{%
\mathbb{D}}}\beta ,$ let $(\Gamma )_{\tau }$ be $V(\Gamma ;\left[ \alpha
,\tau \right] _{\mathbb{D}}).$ That is%
\begin{equation}
(\Gamma )_{\tau }=\sup\nolimits_{\mathbb{D}}\left\{
\sum\limits_{k=1}^{n}\left\vert \Gamma (\tau _{k})-\Gamma (\tau
_{k-1})\right\vert _{\mathbb{D}}:\{\tau _{0},\tau _{1},...,\tau _{n}\}\text{
is a partition of }\left[ \alpha ,\tau \right] _{\mathbb{D}}\right\} .
\label{3}
\end{equation}

Let $\alpha =\alpha _{1}\mathbf{e}_{1}+\alpha _{2}\mathbf{e}_{2},\beta
=\beta _{1}\mathbf{e}_{1}+\beta _{2}\mathbf{e}_{2},\tau =t\mathbf{e}_{1}+s%
\mathbf{e}_{2}$ and $\tau _{k}=t_{k}\mathbf{e}_{1}+s_{k}\mathbf{e}_{2}$ for $%
k=0,1,...n.$

Since $\Gamma $ is a rectifiable $\mathbb{D-}$path, for $i=1,2$ $\gamma
_{i}:[\alpha _{i},\beta _{i}]\rightarrow 
%TCIMACRO{\U{2102} }%
%BeginExpansion
\mathbb{C}
%EndExpansion
$ are rectifiable path.

Let $\alpha _{1}\leq t\leq \beta _{1}$, $\alpha _{2}\leq s\leq \beta _{2}$
and also%
\begin{equation*}
(\gamma _{1})_{t}=\sup \left\{ \sum\limits_{k=1}^{n}\left\vert \gamma
_{1}(t_{k})-\gamma _{1}(t_{k-1})\right\vert :\{t_{0},t_{1},...,t_{n}\}\text{
is a partition of }\left[ \alpha _{1},t\right] \right\} ,
\end{equation*}%
\begin{equation*}
(\gamma _{2})_{s}=\sup \left\{ \sum\limits_{k=1}^{n}\left\vert \gamma
_{2}(s_{k})-\gamma _{2}(s_{k-1})\right\vert :\{s_{0},s_{1},...,s_{n}\}\text{
is a partition of }\left[ \alpha _{1},s\right] \right\} .
\end{equation*}

Then from (\ref{3}) we have%
\begin{equation*}
(\Gamma )_{\tau }=(\gamma _{1})_{t}\mathbf{e}_{1}+(\gamma _{2})_{s}\mathbf{e}%
_{2}.
\end{equation*}

Since $(\gamma _{1})_{t}$ and $(\gamma _{2})_{s}$ are increasing, $(\gamma
_{1})_{t}:[\alpha _{i},\beta _{i}]\rightarrow 
%TCIMACRO{\U{211d} }%
%BeginExpansion
\mathbb{R}
%EndExpansion
$ and $(\gamma _{2})_{s}:[\alpha _{i},\beta _{i}]\rightarrow 
%TCIMACRO{\U{211d} }%
%BeginExpansion
\mathbb{R}
%EndExpansion
$ are bounded variation. So, by Proposition \ref{P2}, $(\Gamma )_{\tau }:%
\left[ \alpha ,\beta \right] _{\mathbb{D}}\rightarrow \mathbb{D}$ is of $%
\mathbb{D}-$bounded variation.

If $f(=f_{1}\mathbf{e}_{1}+f_{2}\mathbf{e}_{2})$ is product-type $\mathbb{D-}
$continuous function on $\{\Gamma \}$ define%
\begin{equation}
\int\limits_{\Gamma }f\left\vert dz\right\vert _{\mathbb{D}}=\int\limits_{%
\left[ \alpha ,\beta \right] _{\mathbb{D}}}f(\Gamma (\tau ))d(\Gamma )_{\tau
}.  \label{4}
\end{equation}

Clearly $f_{1}$ is continuous on $\{\gamma _{1}\}$ and $f_{2}$ is continuous
on $\{\gamma _{2}\}$. If we define%
\begin{equation*}
\int\limits_{\gamma _{1}}f_{1}\left\vert dz_{1}\right\vert
=\int\limits_{\alpha _{1}}^{\beta _{1}}f_{1}(\gamma _{1}(t))d(\gamma
_{1})_{t},
\end{equation*}%
\begin{equation*}
\int\limits_{\gamma _{2}}f_{2}\left\vert dz_{2}\right\vert
=\int\limits_{\alpha _{2}}^{\beta _{2}}f_{2}(\gamma _{2}(s))d(\gamma
_{2})_{s},
\end{equation*}

then for $dz=dz_{1}\mathbf{e}_{1}+dz_{2}\mathbf{e}_{2}$, from (\ref{4}) we
have%
\begin{equation}
\int\limits_{\Gamma }f\left\vert dz\right\vert _{\mathbb{D}}=\left(
\int\limits_{\gamma _{1}}f_{1}\left\vert dz_{1}\right\vert \right) \mathbf{e}%
_{1}+\left( \int\limits_{\gamma _{2}}f_{2}\left\vert dz_{2}\right\vert
\right) \mathbf{e}_{2}.  \label{5}
\end{equation}

If $\Gamma $ is rectifiable $\mathbb{D}-$curve in $\mathbb{BC}$ then denote
by $-\Gamma $ the $\mathbb{D}-$curve defined by $(-\Gamma )(\tau )=\Gamma
(-\tau )$ for $-\beta \preceq _{\mathbb{D}}\tau \preceq _{\mathbb{D}}-\alpha
.$ Also if $c\in \mathbb{BC}$ let $\Gamma +c$ denote the curve defined by $%
(\Gamma +c)(\tau )=$ $\Gamma (\tau )+c$ for $\tau \in \left[ \alpha ,\beta %
\right] _{\mathbb{D}}.$

\begin{proposition}
Let $\Gamma (=\gamma _{1}\mathbf{e}_{1}+\gamma _{2}\mathbf{e}_{2}):\left[
\alpha ,\beta \right] _{\mathbb{D}}\rightarrow \mathbb{BC}$ is a rectifiable 
$\mathbb{D-}$path and suppose that $f(=f_{1}\mathbf{e}_{1}+f_{2}\mathbf{e}%
_{2})$ is a product-type $\mathbb{D-}$continuous function on $\{\Gamma \}.$
Then

$a)$ $\int\limits_{\Gamma }f=-\int\limits_{-\Gamma }f;$

$b)$ $\left\vert \int\limits_{\Gamma }f\right\vert _{\mathbb{D}}\preceq _{_{%
\mathbb{D}}}\int\limits_{\Gamma }\left\vert f\right\vert _{\mathbb{D}%
}\left\vert dz\right\vert _{\mathbb{D}}\preceq _{_{\mathbb{D}}}V(\Gamma
)\sup_{\mathbb{D}}[\left\vert f(z)\right\vert _{\mathbb{D}}:z\in \{\Gamma
\}];$

$c)$ If $c\in \mathbb{BC}$ then $\int\limits_{\Gamma
}f(z)dz=\int\limits_{\Gamma +c}f(z-c)dz.$
\end{proposition}

\begin{proof}
Since $\Gamma $ is rectifiable $\mathbb{D-}$path, for $i=1,2$ $\gamma _{i}$
are rectifiable curve in $%
%TCIMACRO{\U{2102} }%
%BeginExpansion
\mathbb{C}
%EndExpansion
$ and $f_{i}$ are continuous on $\{\gamma _{i}\}.$

Then by Proposition $1.17$ ($\cite{con},$ Chapter $IV$) we have for $i=1,2$

$i)$ $\int\limits_{\gamma _{i}}f_{i}=-\int\limits_{-\gamma _{i}}f_{i};$

$ii)$ $\left\vert \int\limits_{\gamma _{i}}f_{i}\right\vert \leq
\int\limits_{\gamma _{i}}\left\vert f_{i}\right\vert \left\vert
dz_{i}\right\vert \leq V(\gamma _{i})\sup [\left\vert
f_{i}(z_{i})\right\vert :z_{i}\in \{\gamma _{i}\}];$

$iii)$ If $c_{i}\in \mathbb{C}$ then $\int\limits_{\gamma
_{i}}f_{i}(z_{i})dz_{i}=\int\limits_{\gamma
_{i}+c_{i}}f_{i}(z_{i}-c_{i})dz_{i}.$

Let $dz=dz_{1}\mathbf{e}_{1}+dz_{2}\mathbf{e}_{2}$ and $c=c_{1}\mathbf{e}%
_{1}+c_{2}\mathbf{e}_{2}.$ Then using Remark \ref{R4}, equation \ref{5},
Definition \ref{D3}, Proposition \ref{P2} and properties of $\mathbb{D-}$%
modulus we have the required results.
\end{proof}

It is easy to verify that $(\mathbb{BC}$,$d_{\mathbb{D}})$ is a Hyperbolic
Valued Metric Space \cite{Hyp}, where $d_{\mathbb{D}}(x,y)=d_{1}(x_{1},y_{1})%
\mathbf{e}_{1}+d_{2}(x_{2},y_{2})\mathbf{e}_{2}$ for $x=x_{1}\mathbf{e}%
_{1}+x_{2}\mathbf{e}_{2},y=y_{1}\mathbf{e}_{1}+y_{2}\mathbf{e}_{2}\in 
\mathbb{BC}$ and $d_{1},d_{2}$ are usual metric in $%
%TCIMACRO{\U{2102} }%
%BeginExpansion
\mathbb{C}
%EndExpansion
.$ Let $G=G_{1}\mathbf{e}_{1}+G_{2}\mathbf{e}_{2}$ be product-type open set
in $(\mathbb{BC}$,$d_{\mathbb{D}}),$ then $G_{1}$ and $G_{2}$ are open sets
in complex metric space.

\begin{definition}
A product-type function $F$ is called product-type primitive of a
product-type $\mathbb{D-}$continuous function $f$ on a product-type open set 
$G$ if $F^{\prime }(x)=f(x)$ for all $x\in G.$
\end{definition}

\begin{remark}
\label{R7} In the above definition if we take $F=F_{1}\mathbf{e}_{1}+F_{2}%
\mathbf{e}_{2},f=f_{1}\mathbf{e}_{1}+f_{2}\mathbf{e}_{2}$ and $G=G_{1}%
\mathbf{e}_{1}+G_{2}\mathbf{e}_{2},$ then $F_{1},F_{2}$ are primitives of $%
f_{1},f_{2}$ on $G_{1},G_{2}$ respectively.
\end{remark}

The next theorem is the bicomplex analogue of the Fundamental Theorem of
Calculus for line integrals.

\begin{theorem}
\label{T3} Let $G$ be product-type open set in the hyperbolic metric space $(%
\mathbb{BC}$,$d_{\mathbb{D}})$ and let $\Gamma $ be a rectifiable $\mathbb{D-%
}$path in $G$ with initial and end points $\alpha $ and $\beta $
respectively. If $f:G\rightarrow \mathbb{BC}$ is a product-type $\mathbb{D-}$%
continuous function with a product-tye primitive $F:G\rightarrow \mathbb{BC}$%
, then%
\begin{equation*}
\int\limits_{\Gamma }f=F(\beta )-F(\alpha ).
\end{equation*}
\end{theorem}

\begin{proof}
Let $G=G_{1}\mathbf{e}_{1}+G_{2}\mathbf{e}_{2},$ $F=F_{1}\mathbf{e}_{1}+F_{2}%
\mathbf{e}_{2},$ $f=f_{1}\mathbf{e}_{1}+f_{2}\mathbf{e}_{2},$ $\Gamma
=\gamma _{1}\mathbf{e}_{1}+\gamma _{2}\mathbf{e}_{2},$ $\alpha =\alpha _{1}%
\mathbf{e}_{1}+\alpha _{2}\mathbf{e}_{2},$ and $\beta =\beta _{1}\mathbf{e}%
_{1}+\beta _{2}\mathbf{e}_{2}.$

Then for $i=1,2$ $G_{i}$ are open sets in $%
%TCIMACRO{\U{2102} }%
%BeginExpansion
\mathbb{C}
%EndExpansion
$ and $\gamma _{i}$ are rectifiable path in $G_{i}$ with initial and end
points $\alpha _{i}$ and $\beta _{i}$ respectively.

Therefore by Remark \ref{R7} and by Theorem $1.18$ ($\cite{con},$ Chapter $%
IV $)~we have%
\begin{equation}
\int\limits_{\gamma _{i}}f_{i}=F_{i}(\beta _{i})-F_{i}(\alpha _{i})\text{ \
for }i=1,2.  \label{6}
\end{equation}

Then by Remark \ref{R4} we have%
\begin{eqnarray*}
\int\limits_{\Gamma }f &=&\left( \int\limits_{\gamma _{1}}f_{1}\right) 
\mathbf{e}_{1}+\left( \int\limits_{\gamma _{2}}f_{2}\right) \mathbf{e}_{2} \\
&=&(F_{1}(\beta _{1})-F_{1}(\alpha _{1}))\mathbf{e}_{1}+(F_{2}(\beta
_{2})-F_{2}(\alpha _{2}))\mathbf{e}_{2},\text{ by (\ref{6})} \\
&=&F(\beta )-F(\alpha ).
\end{eqnarray*}
\end{proof}

\begin{corollary}
Let $G,$ $\Gamma $ and $f$ satisfy the same hypothesis as in Theorem \ref{T3}%
. If $\Gamma $ is closed curve then%
\begin{equation*}
\int\limits_{\Gamma }f=0.
\end{equation*}
\end{corollary}

\begin{proof}
Let $G=G_{1}\mathbf{e}_{1}+G_{2}\mathbf{e}_{2},$ $f=f_{1}\mathbf{e}_{1}+f_{2}%
\mathbf{e}_{2}$ and $\Gamma =\gamma _{1}\mathbf{e}_{1}+\gamma _{2}\mathbf{e}%
_{2}.$

Then by Corollary $1.22$ ($\cite{con},$ Chapter $IV$) and using Remark \ref%
{R4} we have%
\begin{equation*}
\int\limits_{\Gamma }f=0.
\end{equation*}
\end{proof}

\end{document}